\documentclass[11pt, leqno]{scrartcl}

\usepackage{amssymb,amsmath,amsthm}
\usepackage{latexsym}
\usepackage{t1enc}
\usepackage{lmodern}
\usepackage[utf8]{inputenc}
\usepackage[dvipsnames]{xcolor}
\usepackage{todonotes}
\usepackage{esint}

\newtheorem{theorem}{Theorem}
\newtheorem{proposition}{Proposition}

\newtheorem{corollary}{Corollary}

\newtheorem{lemma}{Lemma}

\usepackage[pdfproducer={LaTeX with hyperref package},pdfpagelayout=SinglePage,bookmarksnumbered,pdftex=true,breaklinks=true,bookmarks=true,linktocpage=true,pdfpagelabels=true,plainpages=false,pagebackref=false]{hyperref} 
\hypersetup{colorlinks=true,linkcolor=blue,citecolor=blue,urlcolor=blue,filecolor=blue}

\newcommand{\Oc}{\mathcal{O}_{\omega_0}}
\renewcommand{\L}{\Lambda}
\newcommand{\Lc}{\L_{\omega_0}}
\newcommand{\abs}[1]{|#1|}
\newcommand{\lap}[1]{\Delta \hspace{-0.15mm} #1 }
\newcommand{\grad}[1]{\nabla \! #1}
 \renewcommand{\epsilon}{\varepsilon}
    \newcommand{\eps}{\varepsilon}
     \newcommand{\norm}[1]{\|#1\|}
      \newcommand{\bilap}[1]{\Delta^{\!2}\hspace{-0.25mm}#1 }
       \newcommand{\Ray}{\mathcal{R}}

\title{On the optimal domain for minimizing the buckling load of a clamped plate}
\date{}
\author{Kathrin Stollenwerk\footnote{Email: stollenwerk@instmath.rwth-aachen.de}}

\begin{document}
\maketitle

\noindent\textbf{Abstract } We prove the existence of an optimal domain for minimizing the buckling load among all, possibly unbounded, open subsets of $\mathbb{R}^n$ ($n\geq 2$) with given measure. Our approach is based on the extension of a 2-dimensional existence result of Ashbaugh and Bucur and on the idea of Alt and Caffarelli to focus on the eigenfunction. 



\section{Introduction}

We consider the following variational problem. Let $\Omega \subset \mathbb{R}^n$ be an open set and define 
\[
  \Ray(v,\Omega) := \frac{\int_\Omega\abs{\lap v}^2dx}{\int_\Omega\abs{\grad v}^2dx}
\]
for $v \in W^{2,2}_0(\Omega)$. If the denominator vanishes, we set $\Ray(v,\Omega) = \infty$. The buckling load of the clamped plate $\Omega$ is defined as
\[
  \L(\Omega) := \min\{ \Ray(v,\Omega): v\in W^{2,2}_0(\Omega)\}. 
\]  
In 1951, Polya and Szegö conjectured that the ball minimizes the buckling load among all open sets of given measure (see \cite{PolyaSzego}). It is still an open question to confirm their conjecture. Up to now, there are only partial results known. 

If there exists a smooth, bounded, connected and simply connected open set $\Omega$ which minimizes the buckling load among all open sets of given measure in $\mathbb{R}^n$,  it is known that $\Omega$ is a ball (see \cite{Willms95,StoWa2015}). 

In \cite{Sto2015}, the existence of an optimal domain for minimizing the buckling load among all opens sets of a given measure which are contained in a sufficiently large ball $B \subset \mathbb{R}^{n}$, $n=2,3$, is proven. However, \cite{Sto2015} does not provide any information about the regularity of the achieved optimal domain. 

Ashbaugh and Bucur proved the existence of a plane optimal domain for minimizing $\L$ in two different settings \cite{BucurAshbaugh}. On the one hand, they prove  the existence of a minimizer in the family of connected and simply connected open stes of given measure in $\mathbb{R}^2$. On the other hand, they find an optimal set $\tilde{\Omega}$ for minimizing a relaxed version of the buckling load among all open sets of given measure in $\mathbb{R}^2$. 

In the present paper, we will adapt a part of the approach by Ashbaugh and Bucur. Therefore, let us briefly summarize their idea.
For $\omega_0>0$ let us denote 
\[
 \Oc := \{\Omega\subset \mathbb{R}^2 : \Omega \mbox{ open}, \abs{\Omega}\leq \omega_0\},
\]
where $\abs{\Omega}$ denotes the $n$-dimensional Lebesgue measure of $\Omega \subset \mathbb{R}^n$. 
Ashbaugh and Bucur start from a minimizing sequence $(\Omega_k)_k \subset \Oc$  and the sequence $(u_k)_k$ of corresponding normalized buckling eigenfunctions $u_k \in W^{2,2}_0(\Omega_k)$. Applying a concentration-compactness lemma they deduce the existence of a limit function $u \in W^{2,2}(\mathbb{R}^2)$ such that 
\begin{equation}\label{eq:intro_1}
  \Ray(u,\mathbb{R}^2) \leq \liminf_{k\to\infty} \Ray(u_k,\Omega_k) = \liminf_{k\to\infty} \L(\Omega_k) = \inf_{\Omega\in\Oc}\L(\Omega). 
\end{equation}
Since $u\in W^{2,2}(\mathbb{R}^2)$, Sobolev's embedding theory implies that $u$ is continuous and the set 
\[
  \tilde{\Omega} := \{x \in \mathbb{R}^2: u(x)\neq 0\}
\]
is an open set. Moreover, the strong $L^2$-convergence of $u_k$ to $u$ implies that $\abs{\tilde{\Omega}}\leq \omega_0$. Hence, $\tilde{\Omega}\in\Oc$.
At this point the authors face the difficulty that their ansatz does not provide any further information about $\tilde{\Omega}$ and $u$ except that $\tilde{\Omega}\in\Oc$ and $u$ is continuous. In particular, they cannot conclude that $u \in W^{2,2}_0(\tilde{\Omega})$. They circumvent that problem by introducing the relaxed Sobolev space $\tilde{W}^{2,2}_0(\tilde{\Omega})$ by
\[
  \tilde{W}^{2,2}_0(\tilde{\Omega}) := \{ v \in W^{2,2}(\mathbb{R}^2): v =0 \mbox{ a.e. in }  \mathbb{R}^2\setminus\Omega\}
\]
and the relaxed buckling load by 
\[ 
  \tilde{\L}(\Omega) := \min_{v \in \tilde{W}^{2,2}_0(\Omega)}\Ray(v,\Omega).
\]
By construction,  $u \in \tilde{W}^{2,2}_0(\tilde{\Omega})$ and, consequently, $\tilde{\Omega}$ minimizes $\tilde{\L}$ in $\Oc$, i.e. 
\[
  \tilde{\L}(\tilde{\Omega}) \stackrel{\eqref{eq:intro_1}}{\leq} \Ray(u,\tilde{\Omega}) \leq \liminf_{k\to\infty} \Ray(u_k,\Omega_k) = \inf_{\Omega\in\Oc}\L(\Omega) = \inf_{\Omega\in\Oc}\tilde{\L}(\Omega), 
\]
where \cite[Theorem 3.1]{BucurAshbaugh} provides the last equation.

In this paper, we will adapt the idea of Ashbaugh and Bucur in \cite{BucurAshbaugh} and extend it to arbitrary dimension. Contrary to their construction via $\tilde{W}^{2,2}$ we prove higher regularity of the limit function $u$. Thereby we follow the idea of Alt and Caffarelli in \cite{AltCaf81}. 
We will find that the first order derivatives of $u$ are  $\alpha$-Hölder continuous  in $\mathbb{R}^n$ for every $\alpha \in (0,1)$. 

Recall (c.f. \cite[Th. 9.1.3]{AdamsHedberg1996} or \cite[Sec. 3.3.5]{HenrotPierre_book}) that for an open set $\Omega\subset \mathbb{R}^n$ and $v \in W^{2,2}(\mathbb{R}^n)$ there holds 
\begin{center}
$v\in W^{2,2}_0(\Omega)$ if $v =\abs{\grad v} = 0 $ pointwise in $\mathbb{R}^n\setminus \Omega$.
\end{center}
Consequently, the Hölder continuity of the first order derivatives of $u$ implies that $u \in W^{2,2}_0(\Omega^\ast)$ for 
\[
  \Omega^\ast := \{x \in \mathbb{R}^n: u(x)\neq 0\ \mbox{ and } \grad{u(x)}\neq 0\}. 
\]
In addition, $\Omega^\ast$ satisfies $\abs{\Omega^\ast} =\omega_0$ and we deduce that $\Omega^\ast$ minimizes the buckling load among all open sets of given measure in $\mathbb{R}^n$.  

Moreover, we will show that the minimizer $\Omega^\ast$ is connected. 

\section{Existence of a minimizer}\label{sec:existence}

 For $\omega_0>0$ we denote the class of admissible sets by 
\[
 \Oc := \{\Omega\subset \mathbb{R}^n : \Omega \mbox{ open}, \abs{\Omega}\leq \omega_0\},
\]
where $\abs{\Omega}$ denotes the $n$-dimensional Lebesgue measure of $\Omega \subset \mathbb{R}^n$, $n\geq 2$.

Our aim is to prove the existence of a set $\Omega^\ast \in \Oc$ which minimizes $\L$ in $\Oc$. In the beginning, we follow the idea of \cite{BucurAshbaugh}.  

Let $\left(\Omega_k\right)_k \in \Oc$ be a minimizing sequence for the buckling load, i.e.
\[
  \lim_{k\to\infty}\L(\Omega_k) = \inf_{\Omega \in \Oc}\L(\Omega) =: \Lc. 
\]

By $u_k \in W^{2,2}_0(\Omega_k)$ we denote the normalized buckling eigenfunction on $\Omega_k$. Hence, $u_k$ satisfies 
\[
   \int\limits_{\Omega_k}\abs{\grad u_k}^2dx =1 \mbox{ and } \L(\Omega_k) = \int\limits_{\Omega_k}\abs{\lap u_k}^2dx
\] 
We now apply the approach by Ashbaugh and Bucur from \cite{BucurAshbaugh} to show that $(u_k)_k$ converges weakly to a limit function  $u$ in $W^{2,2}(\mathbb{R}^n)$.

We will use the following concentration-compactness lemma (see \cite{BucurAshbaugh,Lions1984}) adapted to our setting.  
\begin{lemma}\label{la:ccp}
 Let $(\Omega_k)_k \subset \Oc$ be a minimizing sequence for the buckling load in $\Oc$ and $(u_k)_k$ be the sequence of corresponding eigenfunctions. Then there exists a subsequence $(u_k)_k$ such that one of the three following situations occurs.
  \begin{enumerate}
  \item \textbf{Compactness.} $\exists (y_k)_k\subset \mathbb{R}^n$ such that $\forall\eps>0$, $\exists R <\infty$ and 
  \[
    \forall k \in \mathbb{N} \quad \int\limits_{B_R(y_k)}\abs{\grad u_k}^2dx \geq 1-\eps. 
  \]
  \item \textbf{Vanishing.} $\forall R \in (0,\infty)$ 
  \[
     \lim_{k\to\infty}\sup_{y \in \mathbb{R}^n}\int_{B_R(y)}\abs{\grad u_k}^2dx =0.
  \]
  \item \textbf{Dichotomy.} There exists an $\beta \in (0,1)$ such that $\forall \eps>0 $ there exist two bounded sequences $(u_k^1)_k$, $(u_k^2)_k \subset H^{2,2}(\mathbb{R}^n)$ such that:
  \begin{align}
     \norm{\grad u_k - \grad u_k^1 - \grad u_k^2}_{L^2(\mathbb{R}^2,\mathbb{R}^2)} \leq \delta(\eps) \stackrel{k\to\infty}{\longrightarrow} 0^+, \tag{a} \\
     \abs{\,\int\limits_{\mathbb{R}^n}\abs{\grad u_k^1}^2dx - \beta} \to 0 \quad \mbox{and } \quad  \abs{\,\int\limits_{\mathbb{R}^n}\abs{\grad u_k^2}^2dx - (1-\beta)} \to 0,\tag{b} \\
     \operatorname{dist}(\operatorname{supp}(u_k^1),\operatorname{supp}(u_k^2)) \stackrel{k\to\infty}{\longrightarrow}  \to \infty, \tag{c} \\
     \liminf_{k\to\infty}\left[ \int\limits_{\mathbb{R}^n} \abs{\lap u_k}^2 - \abs{\lap{u_k^1}}^2-\abs{\lap u_k^2}^2dx \right] \geq 0. \tag{d}
  \end{align}
  \end{enumerate} 
\end{lemma}
\begin{proof}
  As mentioned in the proof of \cite[Lemma 3.5]{BucurAshbaugh} the proof is done by considering the concentration function
  \[
   R \to Q_k(R) := \sup_{y\in\mathbb{R}^n}\,\int\limits_{B_R(y)}\abs{\grad u_k}^2dx 
  \]
  for $R \in [0,\infty)$ and following the same steps as in \cite{Lions1984}.
\end{proof}

We will see that for the sequence of eigenfunctions $(u_k)_k$ the case of vanishing and dichotomy cannot occur. Hence, $(u_k)_k$ contains a subsequence, which we again denote by $(u_k)_k$, for which the case of compactness holds true. This compactness will imply the weak convergence of   $u_k$ to a limit function $u$ in $W^{2,2}(\mathbb{R}^2)$. Moreover, the compactness yields that $u_k$ converges to $u$ strongly in $W^{1,2}(\mathbb{R}^n)$.

The case of dichotomy can be disproved in exactly the same way as in \cite{BucurAshbaugh}. For the sake of brevity, we forgo the repetition of this argument. 

In order to disprove the case of vanishing we slightly differ from \cite{BucurAshbaugh}. Nevertheless, we adopt the following lemma \cite[Lemma 3.3]{BucurVarchon} (or \cite[Lemma 6]{Lieb1983}) which is used in \cite{BucurAshbaugh} and which we will apply to disprove the vanishing, as well. 

\begin{lemma}\label{la:trafolemma} 
   Let $(w_k)_k$ be a bounded sequence in $W^{1,2}(\mathbb{R}^n)$ such that $\norm{w_k}_{L^2(\mathbb{R}^n)} = 1$ and $w_k\in W^{1,2}_0(D_k)$ for a $D_k \in \Oc$. There exists a sequence of vectors $(y_k)_k \subset \mathbb{R}^n$ such that the sequence $(w_k(\cdot+y_k))_k$ does not possess a subsequence converging weakly to zero in $W^{1,2}_0(\mathbb{R}^n)$.
\end{lemma}

Now let us assume that for a subsequence of $(u_k)_k$, again denoted by $(u_k)_k$, the case of vanishing occurs. Hence, for every $R>0$ there holds
\begin{equation}\label{eq:vanishing}
  \lim_{k\to\infty} \sup_{y \in \mathbb{R}^n} \int\limits_{B_R(y)}\abs{\grad u_k}^2dx =0.
\end{equation}
Since there holds $\norm{\grad u}_{L^2(\mathbb{R}^n)}=1$ for every $k \in \mathbb{N}$, we obtain for at least one $1\leq l_k \leq n$
\[
   \int\limits_{\mathbb{R}^n} \abs{\partial_{l_k} u_k}^2dx \geq \frac{1}{n}. 
\]
We now consider the sequence $(\partial_{l_k} u_k)_k$. Then $\partial_{l_k} u_k \in W^{1,2}_0(\Omega_k)$ and 
\[
 \frac{1}{\sqrt{n}} \leq \norm{\partial _{l_k} u_k}_{L^2(\mathbb{R}^n)} := c_k. 
\]
The sequence $(v_k)_k$ given by $v_k := c_k^{-1}\partial_{l_k}u_k$ then satisfies the assumptions of Lemma \ref{la:trafolemma}. Consequently, there exists a sequence $(y_k)_k \subset \mathbb{R}^n$ such that the sequence $(v_k(\cdot+y_k))_k \subset W^{1,2}(\mathbb{R}^n)$ does not posses a subsequence which converges weakly to zero in $W^{1,2}(\mathbb{R}^n)$. However, the sequence $(v_k(\cdot+y_k))_k$ is uniformly bounded in $W^{1,2}(\mathbb{R}^n)$ because of the normalization. Hence, there exists a $v \in W^{1,2}(\mathbb{R}^n)$ such that a subsequence of $(v_k(\cdot+y_k))_k$ converges weakly in $W^{1,2}(\mathbb{R}^n)$ to $v$. In particular, there holds 
\[
   v_k(\cdot+y_k) \stackrel{k\to\infty}{\rightharpoonup} v \mbox{ in } W^{1,2}(B_R(0)) \mbox{ for every } R>0 
\]
and 
\[
   v_k(\cdot+y_k) \stackrel{k\to\infty}{\longrightarrow} v \mbox{ in }L^2(B_R(0)) \mbox{ for every } R>0.
\]
Thus, we obtain 
\begin{align*}
  \norm{v}^2_{L^2(B_R(0))} &= \lim_{k\to\infty}\norm{v_k(\cdot+y_k)}^2_{L^2 (B_R(0))}=\lim_{k\to\infty}\frac{1}{c_k^2}\int\limits_{B_R(0)}\abs{\partial_l u_k(x+y_k)}^2dx \\
  &\leq n \lim_{k\to\infty}\int\limits_{B_R(y)}\abs{\grad u_k}^2dx \\
  &\leq n \lim_{k\to\infty} \sup_{y \in\mathbb{R}^n} \int\limits_{B_R(y)}\abs{\grad u_k}^2dx \stackrel{\eqref{eq:vanishing}}{=} 0. 
\end{align*}
Hence, $v =0$ in $L^2(B_R(0))$ and since $v$ is the weak limit of $v_k(\cdot+y_k)$ this is a contradiction to Lemma \ref{la:trafolemma}. Therefore, the case of vanishing cannot occur. 

Consequently, the case of compactness must occur. Following the lines of \cite{BucurAshbaugh} we find that there exists a sequence $(y_k)_k \subset \mathbb{R}^n$ and an $u \in W^{2,2}(\mathbb{R}^n)$ such that 
\begin{equation}\label{eq:weak_conv}
  u_k(\cdot + y_k) \rightharpoonup u \mbox{ in } W^{2,2}(\mathbb{R}^n)
\end{equation}
and, since we are in the compactness case of Lemma \ref{la:ccp}, 
\begin{equation}\label{eq:norm_u}
   \int\limits_{\mathbb{R}^n}\abs{\grad u}^2dx =1.
\end{equation}

From now on, we set
\[
   u_k = u_k(\cdot+y_k)\quad  \mbox{  and  }  \quad \Omega_k = \Omega_k + y_k,
\]
where $(y_k)_k$ is given above. This is possible without loss of generality because of the translational invariance of the buckling load.

We now show that $u_k$ converges strongly to $u$ in $W^{1,2}(\mathbb{R}^n)$. Since this observation will be crucial for constructing an optimal domain in Section \ref{subsec:domain}, we give a detailed proof although we follow the lines of \cite{BucurAshbaugh}.

\begin{lemma}\label{la:strong_W^1,2_convergence}
  There holds 
  \[
     u_k \stackrel{k\to\infty}{\longrightarrow} u \mbox{ in } W^{1,2}(\mathbb{R}^n). 
  \]
\end{lemma} 
\begin{proof}
  We use the notation above. Recall, that $u_k = u_k(\cdot+y_k)$ and $\Omega_k = \Omega_k + y_k $. Then we get from \eqref{eq:norm_u}
  \[
    \int\limits_{\mathbb{R}^n}\abs{\grad u -\grad u_k}^2dx = 2- 2\int\limits_{\mathbb{R}^n}\grad u.\grad u_k\,dx
  \] 
and the weak convergence of $(u_k)_k$ to $u$ in $W^{2,2}(\mathbb{R}^n)$  yields
  \[
     \int\limits_{\mathbb{R}^n}\abs{\grad u -\grad u_k}^2dx \stackrel{k\to\infty}{\longrightarrow} 0. 
  \]
  Thus, $(\grad u_k)_k$ converges to $\grad u$ in $L^2(\mathbb{R}^n)$ and, in particular,  $(\grad u_k)_k$ is a Cauchy sequence in $L^2(\mathbb{R}^n,\mathbb{R}^n)$.
  Now let $l,k \in \mathbb{N}$. Then $u_l-u_k \in W^{2,2}_0(\Omega_l\cup\Omega_k)$ and applying Poincaré's inequality we obtain 
  \begin{align*}
     \int\limits_{\Omega_l\cup\Omega_k}(u_l-u_k)^2dx &\leq \left(\frac{\abs{\Omega_l\cup\Omega_k}}{\omega_n}\right)^\frac{2}{n}\int\limits_{\Omega_l\cup\Omega_k}\abs{\grad(u_l-u_k)}^2dx  \\
     &\leq \left(\frac{2\omega_0}{\omega_n}\right)^\frac{2}{n}\int\limits_{\Omega_l\cup\Omega_k}\abs{\grad(u_l-u_k)}^2dx \stackrel{k\to\infty}{\longrightarrow} 0.
  \end{align*}
  Thus, $(u_k)_k$ is a Cauchy sequence in $L^2(\mathbb{R}^n)$, which converges weakly in $L^2(\mathbb{R}^n)$ to $u$. Consequently, $u_k \stackrel{k\to\infty}{\longrightarrow} u$ in $L^2(\mathbb{R}^n)$. This proves the claim. 
\end{proof}

As a consequence of \eqref{eq:weak_conv} and Lemma \ref{la:strong_W^1,2_convergence}  we obtain that 
\begin{equation}\label{eq:min_R}
  \Ray(u,\mathbb{R}^n) \leq \liminf_{k\to \infty} \Ray(u_k,\Omega_k) = \inf_{\Omega\in\Oc}\L(\Omega).
\end{equation}

 The following proposition summarizes what we have achieved so far.

\begin{proposition}\label{prop:survey_minimizing_sequence}
  Let $(\Omega_k)_k \subset \Oc $ be a minimizing sequence for the buckling load in $\Oc$ and $(u_k)_k$ be the sequence of corresponding normalized eigenfunctions. 
Then there exists a sequence $(y_k)_k \subset \mathbb{R}^n$ such that $u_k(\cdot + y_k)$ is a normalized eigenfunction on $\Omega_k+y_k$ and, denoting $u_k = u_k(\cdot+y_k)$ and $\Omega_k = \Omega_k+y_k$, there exists a subsequence, again denoted by $(u_k)_k$, and an $u \in W^{2,2}(\mathbb{R}^n)$ with 
  \begin{enumerate}
   \item $u$ is normalized by \[
     \int\limits_{\mathbb{R}^n}\abs{\grad u}^2dx =1.
   \]
   \item $ u_k \rightharpoonup u$ in $W^{2,2}(\mathbb{R}^n)$ as $k$ tends to $\infty$.
   \item $u_k \longrightarrow u$ in $W^{1,2}(\mathbb{R}^n)$ as $k$ tends to $\infty$.
   \item There holds $ \Ray(u,\mathbb{R}^n) \leq \inf_{\Omega\in\Oc}\L(\Omega)$.
  \end{enumerate}
\end{proposition}
Recall that in \cite{BucurAshbaugh} only the two dimensional case is considered. Consequently, the limit function $u$ is continuous due to Sobolev's embedding theory. Hence, the set 
\[
  \tilde{\Omega} := \{x \in \mathbb{R}^2: u(x) \neq 0\}
\] 
is an open set and the strong $L^2$-convergence of $u_k$ to $u$ implies that $\tilde{\Omega}\in\Oc$. 

Here, we consider arbitrary dimension. Hence, we need another method to prove regularity of the function $u$. 
Inspired by \cite{AltCaf81}, our approach is based on a careful analysis of the function $u$. This will be done in the next section.

\subsection{Regularity of the limit function}

Our first aim is to show that $u$ has got Hölder continuous first order derivatives. This will be done by using Morrey's Dirichlet Growth Theorem (see Theorem \ref{theo:MDGT}) and a bootstrapping argument based on ideas of Q.~Han and F.~Lin in \cite{HanLin}.

From now on, we consider a minimizing sequence $(\Omega_k)_k \subset \Oc$ such that there holds 
\begin{equation}\label{eq:special_ms}
  \Lc := \inf_{\Omega \in \Oc}\L(\Omega) \leq \L(\Omega_k) \leq \Lc + \frac{1}{k} \mbox{ for every } k \in \mathbb{N}.
\end{equation}
We want to apply the following version of Morrey's Dirichlet Growth Theorem to the first order derivatives of $u$. 

\begin{theorem}\label{theo:MDGT}
 Let $v \in W^{1,2}(\mathbb{R}^n)$ and $0<\alpha\leq 1$ such that for every $x_0 \in \mathbb{R}^n$  and every $0<r\leq r_0$ there holds
 \[
    \int\limits_{B_r(x_0)}\abs{\grad v}^2dx \leq M\cdot r^{n-2+2\alpha}. 
 \] 
 Then $v$ is $\alpha$-Hölder continuous almost everywhere in $\mathbb{R}^n$ and  for almost every $x_1,x_2 \in  \mathbb{R}^n$ there holds
 \[
   \frac{\abs{v(x_1)-v(x_2)}}{\abs{x_1-x_2}^\alpha}\leq C(\alpha)\cdot M.
 \]
\end{theorem}
For a proof of this theorem we refer to \cite[Theorem 3.5.2]{morrey}, e.g.. Hence, we  need a $L^2$-estimate for the second order derivatives of $u$ in every ball $B_r(x_0) \subset \mathbb{R}^n$. 

The following lemmata are preparatory for the proof of Theorem \ref{theo:hölder}, which is the main theorem of this section. Before we start, note that by scaling there holds
\begin{equation}\label{eq:bound_Lc}
  \Lc \leq \left( \frac{\omega_n}{\omega_0}\right)^\frac{2}{n}\L(B_1) \leq C(n,\omega_0), 
\end{equation}
where $B_1$ denotes the unit ball in $\mathbb{R}^n$. 

\begin{lemma}\label{la:reg_1} 
Let $u \in W^{2,2}(\mathbb{R}^n)$ be the limit function according to Proposition \ref{prop:survey_minimizing_sequence} and $0<R\leq 1$. 
  There exists a constant $C=C(n,\omega_0)>0$ such that for every $x_0 \in \mathbb{R}^n$ there holds
  \[
    \int\limits_{B_R(x_0)}\abs{\lap(u-v_0)}^2dx \leq C(n,\omega_0)\left(R^n + \int\limits_{B_R(x_0)}\abs{\grad u}^2dx\right),
  \]
  where $v_0 \in W^{2,2}(B_R(x_0))$ with $v_0-u\in W^{2,2}_0(B_R(x_0))$ and $\bilap v_0=0$ in $B_R(x_0)$.
\end{lemma}
\begin{proof}
The proof is done in three steps. 

\textbf{Step 1}. \quad We choose $x_0  \in \mathbb{R}^n$ arbitrary, but fixed. Let $v_k \in W^{2,2}(B_R(x_0))$ with $v_k-u_k  \in W^{2,2}_0(B_R(x_0))$ and $\bilap v_k =0$ in $B_R(x_0)$.
 If $B_R(x_0)\cap\Omega_k =\emptyset$, $u_k$ and $v_k$ vanish in $B_R(x_0)$. Consequently, we obtain 
\begin{equation}\label{eq:reg_1_1}
  \int\limits_{B_R(x_0)}\abs{\lap(u_k-v_k)}^2dx =0. 
\end{equation}

If $B_R(x_0)\cap\Omega_k \neq\emptyset$, we set 
\[
  \hat{u}_k = \begin{cases}
     u_k, &\mbox{ in } \mathbb{R}^n\setminus B_R(x_0) \\
     v_k, &\mbox{ in } B_R(x_0)
  \end{cases}.
\]
Note that $\Omega_k\cup B_R(x_0)$ is an open set and that $\hat{u}_k \in W^{2,2}_0(\Omega_k\cup B_R(x_0))$. Let us first consider the case $\abs{\Omega_k \cup B_R(x_0)} \leq \omega_0$. Hence, $\Omega_k\cup B_R(x_0) \in \Oc$ and  there holds 
\[
  \Lc = \inf_{\Omega\in\Oc}\L(\Omega) \leq \L(\Omega_k\cup B_R(x_0)) \leq \Ray(\hat{u}_k,\mathbb{R}^n)
\]
since $\hat{u}_k \in W^{2,2}_0(\Omega_k\cup B_R(x_0))$.
Rearranging terms and applying the definition of $\hat{u}_k$ yields 
\begin{align}\label{eq:reg_1_1b}
  \Lc\left(1-\int\limits_{B_R(x_0)}\abs{\grad u_k}^2dx\right) \leq \L(\Omega_k) - \int\limits_{B_R(x_0)}\abs{\lap u_k}^2 - \abs{\lap v_k}^2dx.
\end{align}
Since $v_k - u_k \in W^{2,2}_0(B_R(x_0))$ and $v_k$ is biharmonic in $B_R(x_0)$, there holds 
\[
  \int\limits_{B_R(x_0)}\abs{\lap u_k}^2 - \abs{\lap v_k}^2dx = \int\limits_{B_R(x_0)}\abs{D^2(u_k-v_k)}^2dx. 
\]
We rearrange terms in \eqref{eq:reg_1_1b} and obtain
\begin{equation}\label{eq:reg_1_2}
  \int\limits_{B_R(x_0)}\abs{D^2(u_k-v_k)}^2dx \leq \L(\Omega_k) - \Lc + \Lc\int\limits_{B_R(x_0)}\abs{\grad u_k}^2dx.
\end{equation}

Let us now assume that $\abs{\Omega_k\cup B_R(x_0)}> \omega_0$. Then we set
\begin{equation}\label{eq:mu_k}
  \mu_k := \left(\frac{\abs{\Omega_k}+\abs{B_R}}{\abs{\Omega_k}}\right)^\frac{1}{n}.
\end{equation}
and find that $\mu_k^{-1}\cdot(\Omega_k\cup B_R(x_0)) \in \Oc$. 
Recall that for every $M \subset \mathbb{R}^n$ and $t>0$ the buckling load satisfies 
\begin{equation}\label{eq:scaling_prop}
  \L(M) = t^2\L(t M). 
\end{equation}
Hence, we obtain
\[
  \Lc \leq \L(\mu_k^{-1}(\Omega_k\cup B_R(x_0))) = \mu_k^2\,\L(\Omega_k\cup B_R(x_0)) \leq \mu_k^2\,\Ray(\hat{u}_k,\mathbb{R}^n). 
\]
and, subsequently, 
\begin{equation}\label{eq:reg_1_3}
\mu_k^2\int\limits_{B_R(x_0)}\abs{D^2(u_k-v_k)}^2dx \leq \mu_k^2\,\L(\Omega_k) -\Lc + \Lc\int\limits_{B_R(x_0)}\abs{\grad u_k}^2dx.
\end{equation}
Since $\mu_k>1$, we can collect the estimates \eqref{eq:reg_1_1}, \eqref{eq:reg_1_2} and \eqref{eq:reg_1_3} in the following way: for every $k \in \mathbb{N}$ there holds 
\begin{equation}\label{eq:reg_1_step1}
   \int\limits_{B_R(x_0)}\abs{D^2(u_k-v_k)}^2dx \leq \mu_k^2\,\L(\Omega_k) - \Lc + \Lc\int\limits_{B_R(x_0)}\abs{\grad u_k}^2dx.
\end{equation}

\textbf{Step 2.} We want to understand the limit as $k$ tends to $\infty$ on both sides of \eqref{eq:reg_1_step1}. This needs some preparation. First, recall that we choose a minimizing sequence $(\Omega_k)_k$ such that \eqref{eq:special_ms} holds. Then applying \eqref{eq:scaling_prop} yields
\begin{align*}
   \Lc &\leq \L\left(\left(\frac{\omega_0}{\abs{\Omega_k}}\right)^\frac{1}{n}\Omega_k\right) = \left(\frac{\abs{\Omega_k}}{\omega_0}\right)^\frac{2}{n}\L(\Omega_k)  
    \stackrel{\eqref{eq:special_ms}}{\leq} \left(\frac{\abs{\Omega_k}}{\omega_0}\right)^\frac{2}{n}\left(\Lc+ \frac{1}{k}\right).
\end{align*} 
Rearranging terms yields 
\[
  0 \leq \Lc \left(1-\left(\frac{\abs{\Omega_k}}{\omega_0}\right)^\frac{2}{n}\right) \leq \left(\frac{\abs{\Omega_k}}{\omega_0}\right)^\frac{2}{n}\frac{1}{k}\leq \frac{1}{k}.
\]
Thus, there holds $\abs{\Omega_k}\to \omega_0$ as $k$ tends to $\infty$. This immediately implies that 
\[
  \mu_k \stackrel{k\to\infty}{\longrightarrow} \left(1+\frac{\abs{B_R}}{\omega_0}\right)^\frac{1}{n},
\]
where $\mu_k$ is given in \eqref{eq:mu_k}.
In addition, recall that $u_k \rightharpoonup u$ in $W^{2,2}(\mathbb{R}^n)$ and, therefore, $u_k \rightharpoonup u$ in $W^{2,2}(B_R(x_0))$. Since for all $k \in \mathbb{N}$ there holds 
\[
   \norm{u_k-v_k}^2_{W^{2,2}(B_R(x_0))} \leq 4\int\limits_{B_R(x_0)}\abs{D^2u_k}^2dx \leq 4\norm{u_k}^2_{W^{2,2}(\mathbb{R}^n)} \leq C
\]
and 
\[
  \norm{v_k}_{W^{2,2}(B_R(x_0))} \leq \norm{u_k-v_k}_{W^{2,2}(B_R(x_0))} + \norm{u_k}_{W^{2,2}(B_R(x_0))} \leq C,
\]
there exists a $v_0 \in W^{2,2}(B_R(x_0))$ such that 
\[
   v_k \rightharpoonup v_0 \mbox{ in } W^{2,2}(B_R(x_0)) \mbox{ and }  v_0-u \in W^{2,2}_0(B_R(x_0)). 
\]
Moreover, for every $\phi \in C^\infty_c(B_R(x_0))$ there holds 
\[
 0 = \lim_{k\to\infty}\int\limits_{B_R(x_0)}\lap v_k\lap\phi\,dx = \int\limits_{B_R(x_0)}\lap v_0 \lap\phi\,dx
\]
because $v_k$ is biharmonic in $B_R(x_0)$ for every $k\in\mathbb{N}$ and the weak convergence of $v_k$ to $v_0$ in $W^{2,2}(B_R(x_0))$. Hence, $v_0$ is biharmonic in $B_R(x_0)$. 

\textbf{Step 3.} We take the $\liminf$ on both sides of \eqref{eq:reg_1_step1}. Since $u_k \rightharpoonup u$ in $W^{2,2}_0(B_r(x_0))$, this leads to 
\begin{align*}
  \int\limits_{B_R(x_0)}\abs{D^2(u-v_0)}^2dx &\leq \liminf_{k\to \infty}\int\limits_{B_R(x_0)}\abs{D^2(u_k-v_k)}^2dx  \\
  &\leq \liminf_{k\to \infty}\left(\mu_k^2\,\L(\Omega_k) - \Lc + \Lc\int\limits_{B_R(x_0)}\abs{\grad u_k}^2dx\right) \\
  &= \left(1+\frac{\abs{B_R}}{\omega_0}\right)^\frac{2}{n} \Lc-\Lc+\Lc\int\limits_{B_R(x_0)}\abs{\grad u}^2dx \\
  &\leq C(n,\omega_0)\left(R^n + \int\limits_{B_R(x_0)}\abs{\grad u}^2dx\right).
\end{align*} 
This proves the claim.
\end{proof}
Now let $v_0 \in W^{2,2}(B_R(x_0))$ be the function from Lemma \ref{la:reg_1} and $0<r\leq R$. Then there obviously holds 
\begin{equation}
  \int\limits_{B_r(x_0)}\abs{D^2u}^2dx \leq 2 \int\limits_{B_r(x_0)}\abs{D^2v_0}^2dx + 2 \int\limits_{B_R(x_0)}\abs{D^2(u-v_0)}^2dx
\end{equation}
and applying Lemma \ref{la:reg_1} yields
\begin{equation}\label{eq:reg1}
   \int\limits_{B_r(x_0)}\abs{D^2u}^2dx \leq 2 \int\limits_{B_r(x_0)}\abs{D^2v_0}^2dx +C(n,\omega_0)\left(R^n + \int\limits_{B_R(x_0)}\abs{\grad u}^2dx\right).
\end{equation}
In order to estimate the first summand on the right hand side of the above inequality we cite Lemma 2.1 from \cite{Sto2015}.
\begin{lemma}\label{la:biharm_exten_est}
  Using the notation above there exists a constant $C=C(n)>0$ such that for $0< r\leq R$ there holds 
  \[
    \int\limits_{B_r(x_0)}\abs{D^2v_0}^2 \leq C(n)\left(\frac{r}{R}\right)^n\int\limits_{B_R(x_0)}\abs{D^2u}^2dx.
  \]
  The constant $C$ does not depend on $r, R$ or $x_0$, but on the dimension $n$. 
\end{lemma}
Thus, \eqref{eq:reg1} becomes
\begin{equation}\label{eq:reg2}\begin{split}
  \int\limits_{B_r(x_0)}\abs{D^2u}^2dx \leq C(n)\left(\frac{r}{R}\right)^n&\int\limits_{B_R(x_0)}\abs{D^2u}^2dx\\ &+C(n,\omega_0)\left(R^n + \int\limits_{B_R(x_0)}\abs{\grad u}^2dx\right).
\end{split}\end{equation}
This estimate will be the starting point for the bootstrapping argument which will lead to the Hölder-continuity of the first order derivatives of $u$. 

From \cite[Chapter III, Lemma 2.1]{giaq_mult_int} we cite the next lemma. 
\begin{lemma}\label{la:tech_morrey}
 Let $\Phi$ be a nonnegative and nondecreasing function on $[0,R]$. Suppose that there exist positive constants $\gamma, \alpha, \kappa, \beta$, $\beta<\alpha$, such that for all $0\leq r\leq R \leq R_0$ 
\[
  \Phi(r) \leq \gamma\left[ \left(\frac{r}{R}\right)^\alpha + \delta \right]\Phi(R)+\kappa R^\beta.
\]
Then there exist positive constants $\delta_0 =\delta_0(\gamma,\alpha,\beta)$ and $C=C(\gamma, \alpha,\beta)$ such that if $\delta < \delta_0$, for all $0\leq r\leq R\leq R_0$ we have
\[
  \Phi(r) \leq C \left(\frac{r}{R}\right)^\beta \left[\Phi(R) + \kappa R^\beta\right]  .
\] 
\end{lemma}
The following lemma is based on ideas of \cite[Chapter 3]{HanLin}. It will be the crucial observation for the bootstrapping. 
\begin{lemma}\label{la:bootstrap1}
 Suppose that for each $0\leq r\leq 1$ there holds 
 \[
  \int\limits_{B_r(x_0)} \abs{D^2u}^2dx \leq M\,r^\mu,
 \]
where  $M>0$ and $\mu \in [0,n)$. 
Then there exists a constant $C(n)>0$ such that for each $0\leq r\leq 1$
\[
 \int\limits_{B_r(x_0)}\abs{\grad u}^2dx \leq C(n,M)\,r^\lambda,
\]
where $\lambda = \mu +2$ if $\mu < n-2$ and $\lambda$ is arbitrary in $(0,n)$ if $n-2\leq \mu < n$.
\end{lemma}
\begin{proof}
 Let $0\leq r\leq s \leq 1$. For a function $w \in W^{1,2}(\mathbb{R}^n)$ we set 
 \[
  (w)_{r,x_0} := \fint\limits_{B_r(x_0)}w\,dx = \frac{1}{\abs{B_r(x_0)}}\int\limits_{B_r(x_0)}w\,dx.
 \]
Using this notation we write
\[
\int\limits_{B_r(x_0)}\abs{\grad u}^2dx = \sum_{i=1}^n \int\limits_{B_r(x_0)}\abs{\partial_i u - (\partial_i u)_{s,x_0} + (\partial_i u)_{s,x_0}}^2dx.
\]
Then Young's inequality implies
\begin{align*}
 \int\limits_{B_r(x_0)}\abs{\grad u}^2dx &\leq 2 \sum_{i=1}^n \left( \int\limits_{B_r(x_0)}(\partial_i u)^2_{s,x_0}dx + \int\limits_{B_r(x_0)}\abs{\partial_i u - (\partial_i u)_{s,x_0}}^2dx \right) \\
 &\leq2 \sum_{i=1}^n \left(\abs{B_r} \left(\fint\limits_{B_s(x_0)}\partial_i u\,dx \right)^2 + \int\limits_{B_s(x_0)}\abs{\partial_i u - (\partial_i u)_{s,x_0}}^2dx \right).
\end{align*}
Applying Hölder's and a local version of Poincaré's inequality, we find that
\begin{align*}
 \int\limits_{B_r(x_0)}\abs{\grad  u}^2dx &\leq C(n) \left[ \left(\frac{r}{s}\right)^n \int\limits_{B_s(x_0)}\abs{\grad u}^2dx + s^2\int\limits_{B_s(x_0)}\abs{D^2u}^2dx  \right],
\end{align*}
where the constant $C$ only depends on $n$.
By assumption, we can proceed to
\begin{align*}
 \int\limits_{B_r(x_0)}\abs{\grad u}^2dx &\leq C(n) \left[\left(\frac{r}{s}\right)^n \int\limits_{B_s(x_0)}\abs{\grad u}^2dx +M\, s^{\mu+2} \right].
\end{align*}
Now Lemma \ref{la:tech_morrey} implies that for each $0\leq r\leq s\leq 1$ there holds
\begin{align*}
 \int\limits_{B_r(x_0)}\abs{\grad u}^2dx &\leq C(n)  \left(\frac{r}{s}\right)^\lambda \left[\int\limits_{B_{s}(x_0)}\abs{\grad u}^2dx + M\,s^\lambda  \right].
 \intertext{where $\lambda = \mu +2$ if $\mu< n-2$ and $\lambda$ is arbitrary in $(0,n)$ if $n-2 \leq \mu < n$. Choosing $s=1$, we deduce}
 \int\limits_{B_r(x_0)}\abs{\grad u}^2dx &\leq C(n,M)\,r^\lambda.
\end{align*}
\end{proof}
Now we are able to prove the Hölder continuity of the first order derivatives of $u$. 
\begin{theorem}\label{theo:hölder}
Let $u$ be the limit function $u$ according to Proposition \ref{prop:survey_minimizing_sequence}. 
  The first order derivatives of $u$ according are $\alpha$-Hölder continuous almost everywhere on $\mathbb{R}^n$ for every $\alpha \in (0,1)$.
\end{theorem}
\begin{proof}
  Our aim is to show that for every $x_0 \in \mathbb{R}^n$ and every $0<r\leq 1$ there holds 
  \begin{equation}\label{eq:boot1}
    \int\limits_{B_r(x_0)}\abs{D^2u}^2dx \leq C(n,\omega_0)\,r^{n-2+2\alpha}. 
  \end{equation}
  Then Theorem \ref{theo:MDGT} finishes the proof. Let us choose $x_0 \in \mathbb{R}^n$,  $0<r\leq R\leq 1$ and recall estimate \eqref{eq:reg2}:  
  \[
   \int\limits_{B_r(x_0)}\abs{D^2u}^2dx \leq C(n)\left(\frac{r}{R}\right)^n\int\limits_{B_R(x_0)}\abs{D^2u}^2dx +C(n,\omega_0)\left(R^n + \int\limits_{B_R(x_0)}\abs{\grad u}^2dx\right).
  \]
  We will improve this estimate using a bootstrap argument based on Lemma \ref{la:bootstrap1}. Note that for every $0<r\leq 1$ there holds
  \begin{equation}\label{eq:boot2a}
    \int\limits_{B_r(x_0)}\abs{D^2u}^2dx \leq \norm{u}_{W^{2,2}(\mathbb{R}^n)}^2 =  \norm{u}_{W^{2,2}(\mathbb{R}^n)}^2\,r^0.
  \end{equation}
Then Lemma \ref{la:bootstrap1} implies that for every $0<r\leq 1$ there holds
\begin{equation}\label{eq:boot2}
 \int\limits_{B_r(x_0)}\abs{\grad u}^2dx \leq C(n, \norm{u}_{W^{2,2}(\mathbb{R}^n)})\,r^{\lambda_0},
\end{equation}
where $\lambda_0 \in (0,n)$ if $n=2$ and $\lambda_0 = 2$ if $n\geq 3$. We insert this estimate \eqref{eq:reg2}. Since $R\leq 1$ we obtain 
\[
  \int\limits_{B_r(x_0)} \leq C(n)\left(\frac{r}{R}\right)^2\int\limits_{B_R(x_0)}\abs{D^2u}^2dx + C(n,\omega_0, \norm{u}_{W^{2,2}(\mathbb{R}^n)})\,R^{\lambda_0}
\]
for every $0< r\leq R$. Applying Lemma \ref{la:tech_morrey}, we obtain 
\[
  \int\limits_{B_r(x_0)}\abs{D^2u}^2dx \leq C\,\left(\frac{r}{R}\right)^{\lambda_0}\left(\int\limits_{B_R(x_0)}\abs{D^2u}^2dx +  C(n,\omega_0, \norm{u}_{W^{2,2}(\mathbb{R}^n)})\,R^{\lambda_0}\right)
\]
for every $0<r\leq R$. Choosing $R=1$ leads to 
\begin{equation}\label{eq:boot3}
   \int\limits_{B_r(x_0)}\abs{D^2u}^2dx \leq C(n,\omega_0, \norm{u}_{W^{2,2}(\mathbb{R}^n)})r^{\lambda_0}
\end{equation}
for every $0<r\leq 1$. If $n=2$, this is \eqref{eq:boot1}. 

If $n \geq 3$, \eqref{eq:boot3} is an improvement of estimate \eqref{eq:boot2a}. Recall that here holds $\lambda_0=2$. We again apply Lemma \ref{la:bootstrap1} and obtain for every $0<r\leq 1$
\[
  \int\limits_{B_r(x_0)}\abs{D^2u}^2dx \leq C(n,\omega_0, \norm{u}_{W^{2,2}(\mathbb{R}^n)})r^{\lambda_1},
\]
where $\lambda_1  \in (0,n)$ if $n\in\{3,4\}$ and $\lambda_1 = 4$ if $n\geq 5$. Together with estimate \eqref{eq:reg2} we find that
\[
  \int\limits_{B_r(x_0)}\abs{D^2u}^2dx \leq  C(n)\left(\frac{r}{R}\right)^n\int\limits_{B_R(x_0)}\abs{D^2u}^2dx + C(n,\omega_0, \norm{u}_{W^{2,2}(\mathbb{R}^n)}) R^{\lambda_1}
\]
for every $0<r\leq R\leq 1$. Then Lemma \ref{la:tech_morrey} implies
\[
 \int\limits_{B_r(x_0)}\abs{D^2u}^2dx \leq C\,\left(\frac{r}{R}\right)^{\lambda_1}\left(\int\limits_{B_R(x_0)}\abs{D^2u}^2dx +  C(n,\omega_0, \norm{u}_{W^{2,2}(\mathbb{R}^n)})\,R^{\lambda_1}\right)
\]
and choosing $R=1$ there holds 
\begin{equation}\label{eq:boot4}
   \int\limits_{B_r(x_0)}\abs{D^2u}^2dx \leq C(n,\omega_0, \norm{u}_{W^{2,2}(\mathbb{R}^n)})r^{\lambda_1}
\end{equation}
for every $0<r\leq 1$.  For $n \in \{3,4\}$, estimate \eqref{eq:boot4} and Theorem \ref{theo:MDGT} proves the claim. 

If $n\geq 6$, we repeat the argumentation since \eqref{eq:boot4} is an improvement of \eqref{eq:boot3}. Repeating this process proves the claim after finite many steps for every $n\geq 2$. 
\end{proof}

   Due to Theorem \ref{theo:hölder} the limit function $u$ has a unique representative in $W^{2,2}(\mathbb{R}^n)$ which is continuous in $\mathbb{R}^n$ and which has $\alpha$-Hölder continuous first order derivatives in $\mathbb{R}^n$ for every $\alpha \in (0,1)$.  From now on, we rename this representative as $u$ and focus on this function.

\subsection{The minimizing domain}\label{subsec:domain}

The regularity of $u$, which we achieved in the previous section, enables us to construct an optimal domain for minimizing the buckling load in $\Oc$. 
Recall that there holds (see \eqref{eq:min_R})
\[
\Ray(u,\mathbb{R}^n) \leq \inf_{\Omega\in\Oc}\L(\Omega). 
\]
If $u \in W^{2,2}_0(\Omega^\ast)$ for a suitable set  $\Omega^\ast \in \Oc$, this set $\Omega^\ast$ is the desired minimizer. 
Thus, the challenge is to construct a suitable $\Omega^\ast$. 

Let us define 
\begin{align*}
 \tilde{\Omega} := \{  x \in \mathbb{R}^n: u(x)\neq 0 \} \mbox{ and }
  \hat{\Omega} := \{ x \in \mathbb{R}^n: \abs{\grad u(x)}>0\}.
\end{align*}
Since $u$ and $\grad u$ are continuous on $\mathbb{R}^n$, $\tilde{\Omega}$ and $\hat{\Omega}$ are open sets. 
By definition of $\tilde{\Omega}$ and $\hat{\Omega}$, $u$ vanishes outside $\tilde{\Omega}$ and $\grad u$ vanishes outside $\hat{\Omega}$. 
Now let $(\Omega_k)_k\subset \Oc$  be a minimizing sequence and $(u_k)_k \subset W^{2,2}(\mathbb{R}^n)$ the corresponding sequence of eigenfunctions according to Proposition \ref{prop:survey_minimizing_sequence}.
Then the strong $L^2$-convergence from $u_k$ to $u$ implies
\begin{align*}
  \int\limits_{\mathbb{R}^n}(u-u_k)^2dx & = \int\limits_{\Omega_k}(u-u_k)^2dx + \int\limits_{\Omega_k^c}u^2dx \\
  &= \int\limits_{\Omega_k}(u-u_k)^2dx + \int\limits_{\Omega_k^c\cap\tilde{\Omega}}u^2dx \stackrel{k\to\infty}{\longrightarrow} 0.
\end{align*}
Consequently, there holds $\abs{\Omega_k^c\cap\tilde{\Omega}} \stackrel{k\to\infty}{\longrightarrow} 0$ since $u$ cannot vanish in $\Omega_k^c \cap \tilde{\Omega}$. Analogously, the strong $L^2$-convergence of $\grad u_k$ to $\grad u$ yields $\abs{\Omega_k^c \cap \hat{\Omega} } \stackrel{k\to\infty}{\longrightarrow} 0.$ Now we denote
\begin{equation}\label{eq:Omega^ast}
\Omega^\ast := \tilde{\Omega}\cup\hat{\Omega} = \{ x \in \mathbb{R}^n: u(x) \neq 0 \mbox{ or } \abs{\grad u(x)}\neq 0\}. 
\end{equation}
Note that $\Omega^\ast$ is an open set. In addition, we find that for every $k \in \mathbb{N}$ there holds 
\begin{align*}
  \abs{\Omega^\ast} &= \abs{\Omega^\ast \cap \Omega_k} + \abs{\Omega^\ast\cap\Omega_k^c} \\ 
  &\leq \underbrace{\abs{\Omega_k}}_{\leq\omega_0} +  \abs{\tilde{\Omega}\cap\Omega_k^c}+\abs{\hat{\Omega}\cap\Omega_k^c}.
\end{align*}
Thus, letting $k$ tend to infinity, we obtain $\abs{\Omega^\ast}\leq \omega_0$ and there holds $\Omega^\ast \in \Oc$. 
By construction, $u$ and $\grad u$ vanish in every point in $\mathbb{R}^n\setminus \Omega^\ast$. 

The following corollary guarantees that $u \in W^{2,2}_0(\Omega^\ast)$. For the proof of this corollary we refer to \cite[Th. 9.1.3]{AdamsHedberg1996} or \cite[Sec. 3.3.5]{HenrotPierre_book}. 

\begin{corollary}\label{cor:W^{2,2}}
   Let $\Omega \subset \mathbb{R}^n$ be an arbitrary open set and $v \in W^{2,2}(\mathbb{R}^n)$. If $v$ and its first order derivatives vanish pointwise  in $\mathbb{R}^n\setminus \Omega$, then $u \in W^{2,2}_0(\Omega)$. 
\end{corollary}

Now we can prove our main theorem. 

\begin{theorem}
  The set $\Omega^\ast$ given by \eqref{eq:Omega^ast} minimizes the buckling load $\L$ in $\Oc$. 
\end{theorem}
\begin{proof}
 Recall that there holds 
 \[
   \Ray(u,\mathbb{R}^n) \stackrel{\eqref{eq:min_R}}{\leq} \liminf_{k\to\infty}\Ray(u_k,\Omega_k)  =\inf_{\Omega\in\Oc}\L(\Omega).
 \]
 Since $\Omega^\ast \in \Oc$ and $u \in W^{2,2}_0(\Omega^\ast)$ there holds 
 \[
  \inf_{\Omega\in\Oc}\L(\Omega) \leq \L(\Omega^\ast ) \leq \Ray(u,\Omega^\ast) = \Ray(u,\mathbb{R}^n). 
 \]
 Obviously, this means that 
 \[
    \inf_{\Omega\in\Oc}\L(\Omega)= \L(\Omega^\ast) = \Ray(u,\mathbb{R}^n).
 \]
\end{proof}

Due to the scaling property of the buckling load,  the following corollary holds true.
\begin{corollary}\label{cor:volume}
 Let  $\Omega^\ast \in \Oc$ minimize the buckling load $\L$ in $\Oc$. Then $\Omega^\ast$ satisfies $\abs{\tilde{\Omega}}=\omega_0$.
\end{corollary}

As a consequence of Corollary \ref{cor:volume}, the set $\Omega^\ast$ is connected.

\begin{corollary}\label{cor:connected}
The set $\Omega^\ast$ given by \eqref{eq:Omega^ast} is connected. 
\end{corollary}
\begin{proof}
Let us assume that $\Omega^\ast$ consists of the two connected components  $\Omega_1$ and $\Omega_2$ with $\abs{\Omega_k}>0$ for $k=1,2$. By $u_k$ we denote the eigenfunction $u$ restricted to $\Omega_k$, i.e. 
\[
  u_k := \begin{cases}
    u, &\mbox{ in } \Omega_k\\
    0, &\mbox{otherwise}
  \end{cases}.
\]  
Since $\Omega_k \in \Oc$, the minimality of $\Omega^\ast$ for $\L$ implies
\[
  \L(\Omega^\ast) = \Ray(u,\Omega^\ast)\leq \L(\Omega_1) \leq \Ray(u_1,\Omega_1). 
\]
Rearranging terms and using that $\norm{\grad u}_{L^2(\Omega^\ast)}=1$ we obtain 
\begin{align*}
  \left(  \int\limits_{\Omega_1}\abs{\lap u_1}^2dx + \int\limits_{\Omega_2}\abs{\lap u_2}^2dx\right)\left(1-\int\limits_{\Omega_2}\abs{\grad u_2}^2dx\right) 
  \leq \int\limits_{\Omega_1}\abs{\lap u_1}^2dx.
\end{align*}
Hence, 
\[
  \int\limits_{\Omega_2}\abs{\lap u_2}^2 \leq \L(\Omega^\ast)\,\int\limits_{\Omega_2}\abs{\grad u_2}^2dx \; \Leftrightarrow \; \Ray(u_2,\Omega_2) \leq \L(\Omega^\ast).
\]
Then there holds
\[
  \L(\Omega_2) \leq \Ray(u_2,\Omega_2) \leq \L(\Omega^\ast)
\]
and $\Omega_2$ is a minimizer of $\L$ in $\Oc$. However, since $\abs{\Omega_2} < \omega_0$, this is a contradiction to Corollary \ref{cor:volume}.
\end{proof}

Summing up, we found an optimal domain $\Omega^\ast \in \Oc$ for minimizing the buckling load in $\Oc$. The set $\Omega^\ast$ is open, connected and satisfies $\abs{\Omega^\ast}=\omega_0$. Classical variational arguments show that $u$ solves
\begin{equation*}
  \bilap u + \L(\Omega^\ast) \lap u=0 \mbox{ in } \Omega^\ast.
\end{equation*}

\bigskip 

 \noindent\textbf{Acknowledgement} \; The author is funded by the Deutsche Forschungsgemeinschaft (DFG, German Research Foundation) - project number 396521072. 

%\bibliographystyle{plain}
%\bibliography{../my_literature}%{Stollenwerk}

\end{document}